\documentclass[11pt]{amsart}
\usepackage[margin=1.35in]{geometry}
\usepackage{amscd,amsmath,amsxtra,amsthm,amssymb,stmaryrd,xr,mathrsfs,mathtools,enumerate,commath, comment}
\usepackage{stmaryrd}
\usepackage{xcolor}
\usepackage{commath}
\usepackage{comment}
\usepackage{tikz-cd}
\usepackage{longtable} 
\usepackage{pdflscape} 
\usepackage{booktabs}
\usepackage{hyperref}
\definecolor{vegasgold}{rgb}{0.77, 0.7, 0.35}
\definecolor{darkgoldenrod}{rgb}{0.72, 0.53, 0.04}
\definecolor{gold(metallic)}{rgb}{0.83, 0.69, 0.22}
\hypersetup{
 colorlinks=true,
 linkcolor=darkgoldenrod,
 filecolor=brown,      
 urlcolor=gold(metallic),
 citecolor=darkgoldenrod,
 pdftitle={On the distribution of Alexander polynomials in certain families of closed braids},
 }

\usepackage[all,cmtip]{xy}

\DeclareFontFamily{U}{wncy}{}
\DeclareFontShape{U}{wncy}{m}{n}{<->wncyr10}{}
\DeclareSymbolFont{mcy}{U}{wncy}{m}{n}
\DeclareMathSymbol{\Sh}{\mathord}{mcy}{"58}
\usepackage[T2A,T1]{fontenc}
\usepackage[OT2,T1]{fontenc}

\newtheorem{theorem}{Theorem}[section]
\newtheorem{lemma}[theorem]{Lemma}

\newtheorem*{theorem*}{Theorem}
\newtheorem*{ass*}{Assumption}
\newtheorem{definition}{Definition}[section]

\newtheorem{proposition}[theorem]{Proposition}

\newcommand{\cF}{\mathcal{F}}

\newcommand{\ulambda}{\underline{\lambda}}

\newcommand{\cT}{\mathcal{T}}

\newcommand{\uk}{\underline{k}}

\newcommand{\Z}{\mathbb{Z}}
\newcommand{\p}{\mathfrak{p}}
\newcommand{\Q}{\mathbb{Q}}
\newcommand{\F}{\mathbb{F}}

\newcommand{\cL}{\mathcal{L}}

\newcommand{\cK}{\mathcal{K}}
\newcommand{\cO}{\mathcal{O}}

\newcommand{\op}[1]{\operatorname{#1}}

\newcommand\mtx[4] { \left( {\begin{array}{cc}
 #1 & #2 \\
 #3 & #4 \\
 \end{array} } \right)}

\numberwithin{equation}{section}

\begin{document}

\title[Distribution of Alexander polynomials for closed braids]{On the distribution of Alexander polynomials in certain families of closed braids}

\author[A.~Ray]{Anwesh Ray}
\address[Ray]{Centre de recherches mathématiques,
Université de Montréal,
Pavillon André-Aisenstadt,
2920 Chemin de la tour,
Montréal (Québec) H3T 1J4, Canada}
\email{anwesh.ray@umontreal.ca}

\maketitle

\begin{abstract}
 We study the distribution of arithmetic invariants associated to Alexander polynomials for certain infinite families of links. The families of links we consider arise from braids on a fixed number of strings. We explore analogies with number theory and the distribution of class groups in various families of number fields, setting out new directions in arithmetic topology and arithmetic statistics.
\end{abstract}

\section{Introduction}
\par Given an infinite family of number fields $\cF$ of fixed degree $n$, it is natural to study the distribution of class numbers in $\cF$. Of particular interest is the case when $\cF$ consists of all imaginary quadratic fields $K/\Q$, which are ordered according to discriminant. Such investigations go back to the work of Gauss. Cohen and Lenstra \cite{cohen1984heuristics} make predictions for the distribution of class groups of imaginary quadratic fields. We refer to \cite{davenport1971density, wong1999rank, bhargava2005density,fouvry20074, ellenberg2007reflection,ellenberg2017,holmin2019missing, bhargava2020bounds, klys2020distribution, wang2021moments, pierce2020effective,pierce2021conjecture}, where the study of related questions about the distribution of class groups of number fields has gained significant momentum.

\par There are interesting analogies between algebraic topology and number theory. Such analogies were systematically observed by Mazur in his unpublished notes \cite{mazur1963remarks}. Artin and Mazur \cite{artin2006etale} subsequently set the foundations for the \'etale homotopy theory of number rings. Mazur showed that there are parallels between the arithmetic of primes in number rings, and the geometry of knots embedded in a 3-manifold (cf. \cite{morishita2011knots} for further details). From a Galois theoretic point of view, the analogue of the class group of a number field is the first homology group $H_1(M):=H_1(M;\Z)$ of a closed oriented connected $3$-manifold $M$. The homology groups in question need not be finite. It is still of interest to study their distribution for various natural families of closed $3$-manifolds. More specifically, we consider $3$-manifolds that arise as branched covers of $S^3$. The branch locus for such manifolds is a link, i.e., a union of knots in $S^3$. Given a link $\cL$ in $S^3$, the link group $\op{G}_{\cL}$ is the fundamental group of $X_{\cL}:=S^3\backslash \cL$, with respect to a chosen base-point $x_0$. Let $\xi_\cL:\op{G}_{\cL}^{\op{ab}}\rightarrow \Z$ denote the homomorphism sending each \emph{meridian} to $1$, and $Y_{\cL}\rightarrow X_{\cL}$ be the $\Z$-cover for which $\Gamma:=\op{Gal}(\widetilde{X}_{\cL}/Y_{\cL})$ is identified with the kernel of $\xi_{\cL}$. The Alexander module associated to the cover $Y_{\cL}\rightarrow X_\cL$ is the relative homology group 
\[A_\cL:=H_1(Y_\cL, \cF_{x_0};\Z),\]
where, $\cF_{x_0}$ is the fibre of $x_0$ in $Y_\cL$. Let $\gamma\in \Gamma$ be a choice of generator and fix the isomorphism $\Gamma \xrightarrow{\sim} \Z$, which maps $\gamma$ to $1$. In this way, we identify the group algebra $\Lambda:=\Z[\Gamma]$ with the Laurent series ring $\Z[t^{\pm1}]$. The \emph{Alexander polynomial} associated to $\cL$ is the first fitting ideal for $A_\cL$, when viewed as a module over $\Lambda$. 

\par Throughout, $p$ will denote a fixed odd prime number, and $v_p(\cdot)$ will denote the valuation normalized by setting $v_p(p)=1$. For $r\in \Z_{\geq 1}$, let $\bar{\pi}_r:Y_{\cL,r}\rightarrow X_{\cL}$ be the unique subcover with Galois group $\op{Gal}(Y_{\cL,r}/X_{\cL})\simeq \Z/p^r\Z$. Let $M_{\cL, r}$ denote the \emph{Fox completion} of $Y_{\cL,r}$, in the sense of \cite[Example 2.14]{morishita2011knots}. Note that $\pi_r:M_{\cL, r}\rightarrow S^3$ is a branched cover, and $\pi_r^{-1}(X_{\cL})=Y_{\cL,r}$. Let $\widehat{\Lambda}$ denote the completed algebra $\varprojlim_i \Z_p\left[\Gamma/\Gamma^{p^i}\right]$. Setting $T:=(\gamma-1)$, we identify $\widehat{\Lambda}$ with the formal power series ring $\Z_p\llbracket T\rrbracket$. We regard $\widehat{\Lambda}$ as an algebra over $\Lambda$, upon mapping $t$ to $(1+T)\in \widehat{\Lambda}$. The completed Alexander polynomial is defined as follows
\[\widehat{\Delta}_{\cL}(T):=\Delta_{\cL}(1+T),\]and is well defined up to multiplication by a unit in $\widehat{\Lambda}$. Up to multiplication by a unit in $\widehat{\Lambda}$, $\widehat{\Delta}_{\cL}(T)$ can be expressed as a product $p^{\mu_p(\cL)} g_{\cL}(T)$, where $\mu_p(\cL)\in \Z_{\geq 0}$ and $g_{\cL}(T)$ is a distingushed polynomial (in the sense of \cite[p.115]{washington1997introduction}). The invariant $\mu_p(\cL)$ is referred to as the $\mu$-invariant and $\lambda_p(\cL):=\op{deg} g_{\cL}(T)$ is the $\lambda$-invariant. On the other, hand, if $\Delta_{\cL}(t)=0$, we set $\mu_p(\cL):=\infty$ and $\lambda_p(\cL):=\infty$. Hillman, Matei and Morishita \cite{hillman2005pro} show that if $\Delta_{\cL}(t)\neq 0$ and there $H_1(M_{\cL, r})$ is finite for some $r>0$, then, for some (possibly negative) integer $\nu_p(\cL)$, 
\[v_p\left(\left|H_1(M_{\cL, r})\right|\right)=\mu_p(\cL)p^{r}+\lambda_p(\cL)r+\nu_p(\cL),\] for all large enough values of $n$. The above formula is analogous to Iwasawa's formula for the growth of $p$-parts of class numbers in a $\Z_p$-tower (cf. \cite[Theorem 13.13]{washington1997introduction}). When $\Delta_{\cL}(t)\neq 0$, the invariants $\mu_p(\cL)$ and $\lambda_p(\cL)$ are defined (independent of the assumption on the finiteness of homology groups $H_1(M_{\cL, r})$) and they are of natural significance. These invariants are intrinsically related to the algebraic properties of the $\widehat{\Lambda}$-module $\mathfrak{X}_{\cL}:=H_1(Y_{\cL})\otimes_{\Lambda}\widehat{\Lambda}$, which is a direct summand of the completed Alexander module $\widehat{A_{\cL}}:=A_{\cL}\otimes_{\Lambda} \widehat{\Lambda}$. The module $\mathfrak{X}_{\cL}$ is a torsion module over $\widehat{\Lambda}$ if and only if $\Delta_{\cL}(t)\neq 0$. It is finitely generated $\Z_p$-module if and only if the $\mu$-invariant $\mu_p(\cL)$ vanishes. In this case, the $p$-rank of $H_1(M_{\cL, r})$ is bounded as $r\rightarrow \infty$. Furthermore, if $\mu_p(\cL)=0$, then, in fact $\lambda_p(\cL)$ is equal to the $\Z_p$-rank of $\mathfrak{X}_{\cL}$. In this regard, there is a close analogy between the properties of the module $\mathfrak{X}_{\cL}$ and the classical Iwasawa module (i.e., the module $X=\op{Gal}(L/K_\infty)$ defined on \cite[p.278]{washington1997introduction}).

\par Instead of studying the distribution of the structure of homology groups $H_1(M_{\cL, n})$ themselves (for a fixed prime $p$ and fixed number $n$), we study the distribution of the invariants $\mu_p(\cL)$ and $\lambda_p(\cL)$ for certain natural families of links. Links in $S^3$ are seen to arise very naturally from braids on a given number of strings. The construction involves closing up a braid $b$, by adjoining the starting point of the each string to its ending point, thereby obtaining a link $\cL_b$. Alexander proved that every link arises in this way as a closed braid. The set of isotopy classes of braids on $n\geq 2$ strings has the structure of a group, which is denoted by $B_n$. The braid groups $B_n$ have very interesting representation theoretic properties and arise naturally in various guises. For an overview on the subject of braid groups and their representations, we refer to \cite{kassel2008braid}. For $i$ in the range $1\leq i \leq (n-1)$, we set $\sigma_i=\sigma_i^+$ to denote the braid in $B_n$ prescribed by \cite[Fig 1.9, p.16]{kassel2008braid}. This is the braid for which the $i$-th string overlaps the $(i+1)$-st string, and all other strings are unbraided. The group $B_n$ is generated by $\sigma_1, \dots, \sigma_{n-1}$ subject to the relations
\[\begin{split} & \sigma_i \sigma_j=\sigma_j \sigma_i, \text{ for all }i,j\text{ such that }|i-j|\geq 2,\text{ and }\\
&\sigma_i\sigma_{i+1} \sigma_i=\sigma_{i+1}\sigma_{i} \sigma_{i+1},\text{ for }i=1, \dots, n-2.
\end{split}\]
We define a counting function $|\cdot|:B_n\rightarrow \Z_{\geq 1}$ by setting $|b|$ to denote the minimum length of word in $\langle \sigma_1, \dots, \sigma_{n-1}\rangle$ that realizes $b$. Given $x\in \mathbb{R}_{>0}$, the set $B_n(x)$, which consists of all braids $b\in B_n$ with $|b|\leq x$ is easily seen to be a finite set. It is thus of natural interest to study the variation of the invariants $\mu_{p, b}:=\mu_p(\cL_b)$ and $\lambda_{p, b}:=\lambda_p(\cL_b)$ associated to the Alexander polynomial 
of $\cL_b$, as $b$ varies over $B_n$ and the braids in question are ordered according to $|\cdot |$.
\par Although this question is of natural interest, it is still difficult to obtain an explicit enough description of the Alexander polynomials in the family $B_n$, in order to study this distribution. Instead, in this paper, we restrict our attention to the sub-family of braids $\cF_n\subseteq B_n$ consisting of braids $b$ that can be represented by a word of the form $\sigma_1^{k_1}\sigma_2^{k_2}\dots \sigma_{n-1}^{k_{n-1}}$, where $\uk=(k_1, \dots, k_{n-1})$ is a tuple of non-negative integers. In fact, Proposition \ref{prop injectivity} shows that the isotopy class of the braid $b_{\uk}:=\sigma_1^{k_1}\sigma_2^{k_2}\dots \sigma_{n-1}^{k_{n-1}}$ is determined by $\uk$. In other words, the association sending $\uk$ to $b_{\uk}$ sets up a bijection between $\Z_{\geq 0}^{n-1}$ and $\cF_n$. We now state the main result. Given an integer $N\geq 0$, a composition of $N$ is an ordered tuple of non-negative integers $(\alpha_1, \dots, \alpha_m)$ such that $\sum_i \alpha_i=N$. Given an integer $n\geq 2$, and a non-negative integer $\lambda$, we set 
\begin{equation}\mathfrak{d}_n(\lambda):=\lim_{x\rightarrow \infty}\left(\frac{\#\left\{ \uk=(k_1, \dots, k_{n-1})\in \Z_{\geq 0}^{n-1}\mid \sum_i k_i\leq x\text{ and }\lambda_{p,b_{\uk}}=\lambda\right\}}{\#\left\{\uk=(k_1, \dots, k_{n-1})\in \Z_{\geq 0}^{n-1}\mid \sum_i k_i\leq x\right\}}\right).\end{equation}
In this manuscript, the above limit is shown to exist and its value is computed in terms of explicit combinatorial data, which we now introduce. On the other hand, the $\mu$-invariant $\mu_{p, \uk}$ is shown to vanish for all tuples $\uk$. Denote by $\mathcal{C}_{j, \lambda}$ the set of compositions $\alpha=(p^{a_1}, \dots, p^{a_j})$ of $\lambda$ that are of length $j$, and whose entries are all powers of $p$. The exponents $a_i$ above are all non-negative integers. The following result is unconditional.

\begin{theorem}\label{main thm}
    Let $n\geq 2$ be an integer, and $p$ be an odd prime number. Then, the following assertions hold
    \begin{enumerate}
    \item\label{main thm p1} for each link associated to a braid $b_{\uk}\in \cF_n$, the Alexander polynomial of $\cL(b_{\uk})$ is non-zero and $\mu_{p, b_{\uk}}=0$. 
    \item\label{main thm p2} For $\lambda\in \Z_{\geq 1}$, the limit $\mathfrak{d}_n(\lambda)$ exists and is given by the following explicit combinatorial formula
    \begin{equation}\label{main thm eqn}
    \mathfrak{d}_n(\lambda)=\frac{1}{2^{n-1}p^\lambda}\sum_{j=1}^{n-1}\binom{n-1}{j}\left(1-\frac{1}{p}\right)^j\#\mathcal{C}_{j,\lambda}.
    \end{equation}
    On the other hand, for $\lambda=0$, the limit is given as follows
    \[\mathfrak{d}_n(0)=\frac{1}{2^{n-1}}.\]
    \end{enumerate}
\end{theorem}
We illustrate the above formulae with a concrete example at the end of section \ref{s 3}.

\subsection*{Outlook} The results proven in this paper motivate a number of questions with regard to the variation topological invariants for more general families of braids. The author expects that such questions would potentially lead to interesting counting problems in this topological framework, especially for more general families of braids than those considered here. 

\subsection*{Organization} Including the introduction, the manuscript consists of three sections. In section \ref{s 2}, we introduce preliminary notions and definitions. In section \ref{s 3}, we prove the main result, i.e., Theorem \ref{main thm}.

\subsection*{Acknowledgments}
The author's work is supported by the CRM-Simons postdoctoral fellowship. He thanks Marco Golla for answering his question on Mathoverflow, linked \href{https://mathoverflow.net/questions/435751/alexander-polynomials-for-a-certain-family-of-closed-braids/435781?noredirect=1#comment1123962_435781}{here}.

\section{Completed Alexander polynomials and braid representations}\label{s 2}
\subsection{Analogies between topology and arithmetic}
\par There are many parallels between arithmetic objects and their topological counterparts. For a comprehensive account, we refer to \cite{morishita2011knots}. Objects of fundamental interest on the arithmetic side are the class groups $\op{Cl}(K)$ associated to number fields $K$. On the topological side, the analogous objects of interest are the integral homology groups of closed, oriented and connected $3$-manifolds. Let us explain this analogy in greater detail. Fix an algebraic closure of $K$ and set $K_\emptyset$ to denote the maximal algebraic extension of $K$ in which all finite primes are unramified. On the other hand, note that the Hilbert class field $H_K$ is the maximal abelian extension of $K$ in which all finite primes are unramified. We identify the maximal abelian quotient of $\op{Gal}(K_\emptyset/K)$ with $\op{Gal}(H_K/K)$, which according to class field theory, is canonically identified with the class group $\op{Cl}(K)$. On the other hand, let $\pi:\widetilde{M}\rightarrow M$ be the universal cover of a closed, oriented and connected $3$-manifold $M$. After fixing a base-point of $M$, the Galois group $\op{Gal}(\tilde{M}/M)$ of deck transformations of the covering map $\pi$ is identified with the fundamental group $\pi_1(M)$. The abelianization of $\pi_1(M)$ is the homology group $H_1(M)$, and therefore, $H_1(M)$ is isomorphic to the group of deck transformations of the maximal abelian subcover $\pi^{\op{ab}}: \widetilde{M}^{\op{ab}}\rightarrow M$ of the universal cover. The cover $\widetilde{M}^{\op{ab}}$ is the analogue of the Hilbert class field, and $H_1(M)$ is thus the analogue of the class group.

\par We elaborate on some further aspects of the analogy between number rings and $3$-manifolds. Let $K$ be a number field, with ring of integers $\cO_K$, denote by $\mathfrak{X}$ the scheme $\op{Spec}\cO_K$. Given an \'etale constructible sheaf $\mathfrak{M}$ on $\mathfrak{X}$, with dual sheaf $\mathfrak{M}^\vee$, the \emph{Artin-Verdier duality theorem} states that there is a perfect pairing of abelian groups
\[\hat{H}^i(\mathfrak{X}, \mathfrak{M}^\vee)\times \op{Ext}_{\mathfrak{X}}^{3-i}(\mathfrak{M}, \mathbb{G}_{m, \mathfrak{X}})\rightarrow \hat{H}^3(\mathfrak{X}, \mathbb{G}_{m, \mathfrak{X}})\xrightarrow{\sim} \Q/\Z.\]
Here, $\mathbb{G}_{m, \mathfrak{X}}$ is the \'etale sheaf of multiplicative units on $\mathfrak{X}$, and $\hat{H}^i(\mathfrak{X}, \mathfrak{M})$ is the modified \'etale cohomology groups of $\mathfrak{M}$. We refer to \cite[Theorem 1.8]{milne2006arithmetic} for further details. On the topological side of the analogy, let $M$ be an oriented and closed $3$-manifold. The Poincare duality theorem asserts that there is a perfect cap-product pairing 
\[H^i(M)\times H_3(M)\rightarrow H_{3-i}(M),\]which induces a duality isomorphism 
$H^i(M)\xrightarrow{\sim} H_{3-i}(M)$ on taking the cap product with a fundamental class (cf. \cite[Theorem 3.30]{hatcheralgebraic}).

\par This analogy extends to prime ideals and their ramification behavior. Note that any non-zero prime ideal $\p$ of $\cO_K$ is maximal. Denote by $\F_{\p}$ the finite residue field $\cO_K/\p$. The $0$-dimensional scheme $\op{Spec}\F_\p$ is the \'etale analogue of the Eilenberg-Maclane space $K(\widehat{\Z}, 1)$. On the other hand, $S^1$ is the Eilenberg-Maclane space $K(\Z,1)$. The reduction map induces an immersion of schemes $\op{Spec}\F_{\p}\hookrightarrow \op{Spec}\cO_K$. The topological analogue of a non-zero prime in a number ring is a knot $\cK$ in a $3$-manifold $M$. A knot in $M$ is by definition an embedding $S^1\hookrightarrow M$ which is locally flat, i.e., locally homeomorphic to the embedding $\mathbb{R}\rightarrow \mathbb{R}^3$ mapping $x$ to $(x, 0, 0)$. By abuse of notation the knot $\cK$ will denote the image of such a locally flat embedding. Two knots are equivalent if there is an isotopy $f_t:M\rightarrow M$, $0\leq t\leq 1$, taking one knot to the other. A link $\cL$ is a disjoint union of knots $\cK_1\cup \dots \cup \cK_m$. The main objects of study shall be links $\cL$ contained in the $3$-sphere $S^3$. We shall assume that $\cL$ is oriented, i.e., each component knot $\cK_i$ of $\cL$ is provided an orientation. Given two distinct components $\cK_i$ and $\cK_j$, the linking number is denoted $\ell_{i,j}=\ell(\cK_i, \cK_j)$ (cf. \cite[p.11]{kawauchi1996survey} for a precise definition). Set $X_\cL$ to denote the knot complement $S^3\backslash \cL$, and fix a base-point $x_0$ in $X_\cL$. The link group $\op{G}_{\cL}$ is the based fundamental group $\pi_1(X_{\cL};x_0)$. 
\par Let $K$ be a number field and $\p$ be a non-zero prime ideal in $\cO_K$. Let $K_\p$ be the completion of $K$ at $\p$, and $\cO_\p$ be the valuation ring of $K_\p$. The inclusion $\cO_\p\hookrightarrow K_\p$ induces a map of schemes
\[j_\p: \op{Spec}K_\p\hookrightarrow \op{Spec} \cO_\p,\] mapping $\op{Spec}K_\p$ to the generic point of $\op{Spec} \cO_\p$.
The reduction map $\cO_\p\rightarrow \F_\p$ induces a closed immersion 
\[i_\p: \op{Spec} \F_\p\hookrightarrow \op{Spec}\cO_\p,\] mapping $\op{Spec} \F_\p$ to the special point of $\op{Spec}\cO_\p$.
\par On the topological side, let $\cL$ be a link in $S^3$ and $\cK$ be a component knot of $\cL$. Let $\mathcal{V}$ be the closure of a tubular neighbourhood of $\cK$ (in the sense of \cite[p.137, ll.7-11]{lee2000introduction}). Assume that $\mathcal{V}\backslash \cK$ is contained in $X_\cL$. The diagram \[\op{Spec} K_{\p} \hookrightarrow \op{Spec} \cO_{\p}\hookleftarrow \op{Spec} \F_{\p}\] is analogous to 
\[\partial \mathcal{V} \hookrightarrow \mathcal{V}\hookleftarrow \mathcal{K},\]where $\partial \mathcal{V}$ is the boundary of $\mathcal{V}$. Note that $\partial\mathcal{V}$ is homeomorphic to $S^1\times \cK$. The \emph{longitude} $\sigma_\cK$ is a class in $\pi_1(\partial \mathcal{V})$ generated by an oriented loop of the form $\{a\}\times \cK$. The orientation is chosen to coincide with the orientation of $\cK$. Note that $\mathcal{V}$ deformation retracts onto $\cK$ and this induces an isomorphism $\pi_1(\mathcal{V})\xrightarrow{\sim} \pi_1(\cK)$. Via this isomorphism, the longitude $\sigma_\cK$ maps to a generator of $\pi_1(\cK)$. Thus, the longitude is an analogue of the arithmetic Frobenius element $\sigma_{\p}$, which is a topological generator of $\op{Gal}(\bar{\F}_{\p}/\F_{\p})$. On the other hand, a \emph{meridian} $\tau_{\cK}\in\pi_1(\partial \mathcal{V})$ is the homotopy class of an oriented loop $S^1\times \{b\}$, for $b\in \cK$. We note that there are two choices depending on the orientation, and we make one choice for each component knot. The kernel of the composed map $\pi_1(\partial \mathcal{V})\rightarrow \pi_1(\mathcal{V})\xrightarrow{\sim} \pi(\cK)$ is called the inertia group at $\cK$, and is denoted by $I_{\cK}$. The meridian $\tau_{\cK}$ is a generator of $I_{\cK}$. We note that $\pi_1(\partial \mathcal{V})$ is generated by $\sigma_\cK$ and $\tau_\cK$ subject to the single relation $\sigma_\cK \tau_\cK \sigma_\cK^{-1}\tau_{\cK}^{-1}$. The inclusion of $\mathcal{V}$ into $X_{\cL}$ induces a natural map $\pi_1(\mathcal{V})\rightarrow \pi_1(X_{\cL})=\op{G}_{\cL}$, the image of which is referred to as the \emph{decomposition group} at $\cK$. 
\par Let $K$ be a number field and $S$ be a finite set of primes of $K$. Denote by $K_S$ the maximal extension of $K$ in $\bar{K}$, in which all primes $v\notin S$ are unramified. The extension $K_S$ is a Galois extension of $K$, and is infinite if $S\neq \emptyset$. The Galois group $\op{G}_{K, S}:=\op{Gal}(K_S/K)$ is the maximal quotient of $\op{Gal}(\bar{K}/K)$ which is unramified outside $S$. The group $\op{G}_{K,S}$ is analogous to the Galois group $\op{G}_{\cL}$. The components of $\cL$ are analogous to the primes in $S$, and the group $\op{G}_{\cL}$ is the Galois group of $S^3$ the maximal cover which is unbranched away from $\cL$. 

\subsection{The completed Alexander polynomial}
\par Let $\cL$ be a link in $S^3$ with component knots $\cK_1, \dots, \cK_m$. Denote by $\pi_{\cL}:\widetilde{X_{\cL}}\rightarrow X_{\cL}$ be the universal cover of $X_{\cL}$, and $\op{G}_{\cL}:=\op{Aut}(\widetilde{X_{\cL}}/X_{\cL})$ the Galois group, consisting of deck transformations for the covering map $\pi_{\cL}$. For each component link $\cK_i$ of $\cL$, let $\tau_i\in \op{G}_{\cL}$ be the class generated by the meridian for $\cK_i$. Let $\op{G}_{\cL}^{\op{ab}}$ be the abelianization of $\op{G}_{\cL}$. By the Hurewicz theorem, $\op{G}_{\cL}^{\op{ab}}$ is identified with $H_1(X_\cL;\Z)$. Let $t_i$ denote the image of $\tau_i$ in $H_1(X_\cL; \Z)$ with respect to the quotient map \[\op{G}_{\cL}\rightarrow \op{G}_{\cL}^{\op{ab}}\xrightarrow{\sim} H_1(X_\cL; \Z).\]
We note that $H_1(X_{\cL}; \Z)$ is a free abelian group with $\Z$-basis $t_1, \dots, t_m$. Recall that we defined a map $\xi_{\cL}:\op{G}_{\cL}\rightarrow \Z$ which factors through the quotient $\op{G}_{\cL}^{\op{ab}}$, and sends each meridian $t_i$ to $1\in \Z$. Recall that $Y_\cL\rightarrow X_\cL$ is the cover cut out by $\xi_{\cL}$. The Alexander module $A_{\cL}$ admits a natural action of $\Gamma:=\op{Gal}(Y_{\cL}/X_{\cL})$, and thus, is a module over $\Lambda$. In order to define the \emph{Alexander polynomial} associated with the Alexander module $A_{\cL}$, we recall some prequisite notions from commutative algebra. 

\par Let $M$ be a finitely generated $\Lambda$-module, and $m_1, \dots, m_d$ be a finite set of generators of $M$. For $i$ in the range $1\leq i\leq d$, the $i$-th coordinate generator $e_i$ of $\Lambda^d$ is the tuple $e_i=(0,0, \dots, 0, 1, 0, \dots, 0)$ with $1$ in the $i$-th position and $0$ at all other positions. Let $d_1:\Lambda^d\rightarrow M$ be the map of $\Lambda$-modules mapping $e_i$ to $m_i$ for $i=1, \dots, d$. Pick generators $n_1, \dots, n_q$ for the kernel of $d_1$, and let $d_2:\Lambda^q\rightarrow \Lambda^d$ be the map which takes the $i$-th coordinate vector to $n_i$, and $A$ be the matrix represented by $d_2$. The Alexander polynomial $\Delta_M(t)$ is a generator of the ideal generated by all $(d-1)\times (d-1)$ submatrices of $A$. We note that $\Delta_M(t)$ is only well-defined up to multiplication by a unit in $\Lambda$, i.e.,  up to multiplication by an element of the form $\pm t^{k}$, where $k\in \Z$. We set $\Delta_{\cL}(t):=\Delta_{A_{\cL}}(t)$, and refer to $\Delta_{\cL}(t)$ as the Alexander polynomial of $\cL$. 
Let $\widehat{\Lambda}$ denote the completed algebra $\varprojlim_i \Z_p\left[\Gamma/\Gamma^{p^i}\right]$. Setting $T:=(\gamma-1)$, we identify $\widehat{\Lambda}$ with the formal power series ring $\Z_p\llbracket T\rrbracket$. We regard $\widehat{\Lambda}$ as an algebra over $\Lambda$, upon mapping $t$ to $(1+T)\in \widehat{\Lambda}$. We recall that the completed Alexander polynomial (at the prime $p$) is defined as follows
\[\widehat{\Delta}_{\cL}(T):=\Delta_{\cL}(1+T).\]
Assume that $\widehat{\Delta}_{\cL}(T)\neq 0$. We may factor the power series $\widehat{\Delta}_{\cL}(T)$ as $p^{\mu} g(T)$, where $\mu=\mu_p(\cL)$ is the $\mu$-invariant, and $p\nmid g(T)$. We may write the power series expansion of $g(T)$ as follows\[g(T)=a_0 +a_1T+a_2T^2+\dots+a_i T^i+
\dots,\] and we note that the $\lambda$-invariant is computed as follows
\begin{equation}\label{lambda min}\lambda_p(\cL)=\op{min}\{i\mid p\nmid a_i\}.\end{equation}
In particular, it is easy to see that $\lambda_p(\cL)\geq \op{ord}_{T=0} \widehat{\Delta}_{\cL}(T)$. Set $r:=\op{ord}_{T=0} \widehat{\Delta}_{\cL}(T)$, and write $g(T)$ as follows
\[g(T)=T^r(a_r+a_{r+1}T+\dots+a_{r+i}T^i+\dots).\] Then, we find that $\lambda_p(\cL)= \op{ord}_{T=0} \widehat{\Delta}_{\cL}(T)$ if and only if $p\nmid a_r$.
\subsection{The braid group on $n$-strings}
\par In order to study statistical problems associated to families of links, it conveniences us to discuss the relationship between links and braids. We begin by recalling the notion of a braid on $n$-strings (or strands), where $n\geq 2$ is an integer. Let $I$ denote the interval $[0,1]$. A \emph{braid} is a subset $b$ of $\mathbb{R}^2\times I$, which consists of $n$ strings joining \[b\cap \left(\mathbb{R}^2\times \{0\}\right)=\{(1,0,0), (2,0,0), \dots, (n, 0, 0)\}\] to \[b\cap \left(\mathbb{R}^2\times \{1\}\right)=\{(1,0,1), (2,0,1), \dots, (n, 0, 1)\}.\] Furthermore, it is required that the projection $\mathbb{R}^2\times I\rightarrow I$ maps each string homeomorphically onto $I$. For a more precise definition, we refer to \emph{loc. cit.} Thus, each string meets the slice $\mathbb{R}^2\times \{t\}$ in exactly one point. Associated with a braid $b$, there is an underlying permutation $s_b\in S_n$. This permutation is prescribed so that $(i,0,0)$ is joined by a string to $(s_b(i), 0,1)$. The permutation $s_b$ is referred to as the \emph{underlying permutation of $b$}. Two braids $b$ and $b'$ are said to be \emph{isotopic} if there is a continuous map $F: b\times I \rightarrow \mathbb{R}^2\times I$, such that for each $s\in I$, $b_s=F(b\times \{s\})$ is a braid, $b_0=b$ and $b_1=b'$. The isotopy $F$ is referred to as an \emph{ambient isotopy}. Given two braids $b_1, b_2$, the composite $b_1 b_2$ is the set of points $(x, y,t)\in \mathbb{R}^2\times I$ such that $(x,y,2t)\in b_1$ for $t\in [0, 1/2]$, and $(x,y,2t-1)\in b_2$ for $t\in [1/2,1]$. This prescribes a well defined composition on the set of isotopy classes of braids on $n$ strings, which we denote by $B_n$. Moreover, the composition on $B_n$ is invertible, and $B_n$ forms a group. 
\par For an integer $i$ in the range $1\leq i\leq (n-1)$, let $\sigma_i=\sigma_i^+$ be the braid in $B_n$ prescribed by \cite[Fig 1.9, p.16]{kassel2008braid}. The group $B_n$ consists of $(n-1)$ generators $\sigma_1, \dots, \sigma_{n-1}$ subject to the \emph{braid} relations
\[\begin{split} & \sigma_i \sigma_j=\sigma_j \sigma_i, \text{ for all }i,j\text{ such that }|i-j|\geq 2,\text{ and }\\
&\sigma_i\sigma_{i+1} \sigma_i=\sigma_{i+1}\sigma_{i} \sigma_{i+1},\text{ for }i=1, \dots, n-2.
\end{split}\]
Closing up the braid $b$ gives rise to an oriented link $\cL_b$ contained in the torus $V=D\times S^1$, where $D\subset \mathbb{R}^2$ is the closed disk containing the set of $n$ points $Q:=\{(1,0), (2,0), \dots, (n,0)\}$. We refer to an oriented link of the form $\cL_b$ as a closed braid. Alexander's theorem states that all oriented links in $S^3$ are isotopic to closed braids. The minimum number of strings for the braid in such a Each component of $\cL_b$ corresponds to a cycle of the underlying permutation $s_b$. Thus, the link $\cL_b$ is a knot if and only if $s$ consists of a single cycle. The braid index of an oriented link is the minimum number of strings required to obtain a closed braid representation of a link. The Alexander polynomial associated with $\cL_b$ is denoted by $\Delta_b(t)$ and is well defined up to multiplication by a unit $\pm t^{r}$ for $r\in \Z$. In order to explicitly describe the Alexander polynomial, we briefly introduce the reduced Burau representation $\bar{\psi}_{n}:B_n\rightarrow \op{GL}_{n-1}(\Lambda)$. For a more comprehensive account, we refer to \cite[Section 3.3]{kassel2008braid}. We shall try and make our notation consistent with that of \emph{loc. cit.}
\par The braid group $B_n$ is the mapping class group of the disk $D$ with $n$ marked points $Q=\{(1,0), \dots, (n, 0)\}$. Set $\Sigma:=D\backslash Q$, given a loop $\eta\in \pi_1(\Sigma)$, let $w(\eta)$ be the sum of winding numbers around all points in $Q$ considered in the anti-clockwise direction. Consider the homomorphism 
\[\varphi: \pi_1(\Sigma)\rightarrow \{t^r\}_{r\in \Z},\]mapping a loop $\eta$ to $t^{-w(\eta)}$. The kernel of $\varphi$ defines an infinite cyclic cover $\pi:\widetilde{\Sigma}\rightarrow \Sigma$. Fix a base point $p\in \partial D$, and set $\widetilde{H}:=H_1(\widetilde{\Sigma}, \pi^{-1}(p); \Z)$. The module $\widetilde{H}$ is a free $\Lambda$-module of rank $n$, and consider the natural action of the mapping class group $B_n$ on $\widetilde{H}$. Any element of $g\in B_n$ gives rise to a $\Lambda$-module automorphism $\psi_n(g)$ of $\widetilde{H}$. Choose a $\Lambda$-basis of $\widetilde{H}$, and thus express $\psi_n(g)$ as an element of $\op{GL}_n(\Lambda)$. After a choice of a suitable $\Lambda$-basis of $\widetilde{H}$, the representation $\psi_n:B_n\rightarrow \op{GL}_n(\Lambda)$ is referred to as the Burau representation. This representation can be explicitly described as follows. Let $n\geq 3$ and for $i$ in the range $1\leq i\leq n-1$, let $U_i$ be the $n\times n$ matrix 
\[U_i:=\left( {\begin{array}{cccc}
 \op{I}_{i-1} & 0 & 0 & 0 \\
 0 & (1-t) & t & 0 \\
 0 & 1 & 0 & 0 \\
  0 & 0 & 0 & \op{I}_{n-i-1}, \\
 \end{array} } \right)\]
where $\op{I}_k$ is the $k\times k$ identity matrix. We clarify that when $k=0$, $\op{I}_k$ is the empty matrix, and thus, 
\[U_1:=\left( {\begin{array}{ccc}
 (1-t) & t & 0 \\
 1 & 0 & 0 \\
 0 & 0 & \op{I}_{n-2}, \\
 \end{array} } \right)\text{ and } U_{n-1}:=\left( {\begin{array}{ccc}
 \op{I}_{n-2} & 0 & 0 \\
 0 & (1-t) & t \\
 0 & 1 & 0. \\
 \end{array} } \right)\]
 The representation $\psi_n$ takes $\sigma_i$ to $U_i$. It is easy to check that the braid relations are satisfied for the matrices $\{U_i\mid 1\leq i\leq n-1\}$. The Burau representations $\{\psi_n\}_{n\geq 1}$ are compatible with respect to natural inclusions $\iota: B_n\hookrightarrow B_{n+1}$. For $b\in B_n$, we have that
 \[\psi_{n+1}(\iota(b))=\mtx{\psi_n(b)}{\ast}{0}{1}.\] The representation $\psi_n$ is not irreducible, and $\psi_n^{\op{ss}}\simeq \bar{\psi}_n\oplus \mathbf{1}$, where $\bar{\psi}_n:B_n\rightarrow \op{GL}_{n-1}(\Lambda)$ is the reduced Burau representation. Let $V_i$ be the matrix $V_i=\bar{\psi}_n(\sigma_i)$ for $i$ in the range $1\leq i \leq (n-1)$. The matrices $V_i$ are given as follows
 \[V_1:=\left(\begin{array}{ccc}
 -t & 0 & 0 \\
 1 & 1 & 0 \\
 0 & 0 & \op{I}_{n-3}\\
 \end{array}\right), V_{n-1}:=\left(\begin{array}{ccc}
 \op{I}_{n-3} & 0 & 0 \\
 0 & 1 & t \\
 0 & 0 & -t \\
 \end{array}\right),\]
 and for $i$ in the range $2\leq i\leq n-2$, 
 \[V_i:=\left(\begin{array}{ccccc} 
 \op{I}_{i-2} & 0 & 0 & 0 & 0\\
 0 & 1 & t & 0 & 0\\
 0 & 0 & -t & 0 & 0\\
 0 & 0 & 1 & 1 & 0 \\
 0 & 0 & 0 & 0 & \op{I}_{n-i-2}. \\
 \end{array}\right)\]
 While for $n=2$, the reduced Burau representation $\bar{\psi}_2: B_2\rightarrow \op{GL}_1(\Lambda)$ maps $\sigma_1$ to $-t$. 
\par Burau \cite{burau1935zopfgruppen} showed that the Alexander polynomial of the link obtained from closing up a braid $b$ on $n$ strings is given by
\begin{equation}\label{burau equation}\Delta_{b}(t)=\frac{(1-t)}{(1-t^n)}\op{det}\left(\op{I}_{n-1}-\bar{\psi}_n(b)\right).\end{equation}
When $n\geq 5$, it is known that the kernel of the representation $\bar{\psi}_n$ is non-trivial (cf. \cite[Theorem 3.3]{kassel2008braid}), and thus the Alexander polynomial may indeed by identically $0$.

\section{Statistics for infinite families of braids}\label{s 3}

\par Throughout this section, $p$ will denote a fixed odd prime number. Let $n\geq 2$ be an integer and $B_n$ be the braid group on $n$ strings. We study the variation of topological invariants associated with a suitably chosen family of braids. Recall that the braid group $B_n$ is a quotient of the free group $\langle \sigma_1, \dots, \sigma_{n-1}\rangle$, by the braid relations. We define a counting function $|\cdot |:\langle \sigma_1, \dots, \sigma_{n-1}\rangle\rightarrow \Z_{\geq 0}$, by setting $|\sigma_{i_1}^{a_1}\dots \sigma_{i_t}^{a_t}|:=\sum_{j=1}^t |a_j|$. In other words, given a word, the counting function simply denotes the length. Alternatively, this function counts the number of crossings in the associated braid diagram. This induces a counting function $|\cdot|:B_n\rightarrow \Z_{\geq 0}$, with $|b|$ defined as the minimum value of $|\sigma|$, as $\sigma$ ranges over all words in $\langle \sigma_1, \dots, \sigma_{n-1}\rangle$ that represent $b$. Given any real number $x>0$, the set $B_n(x):=\{b\in B_n\mid |b|\leq x\}$ is finite. This is easy to see, since there are only finitely many words in the free group for which the word length is bounded by $x$. Thus, after taking the quotient by the braid relations, there are only finitely many equivalence classes of such words. 

\par Let $\cF_n$ be the family of braids $b$ on $n$ strings that admit a presentation by a word of the form $\sigma_1^{k_1}\dots \sigma_{n-1}^{k_{n-1}}$, where $\underline{k}=(k_1, \dots, k_{n-1})$ is an $(n-1)$ tuple of non-negative integers. We specialize to this special family since we have an explicit description of Alexander polynomials in this natural family. We let $b_{\underline{k}}$ denote the associated braid in $B_n$. Note that $b_{\uk}$ is thus an isotopy class of braids.

\par Given an integer $r$, let $\mathcal{T}_r$ be the the torus link on $2$ strings which is the closed braid associated to $\sigma_1^r$. Note that $\mathcal{T}_r$ is a knot if and only if $r$ is odd. The Alexander polynomial of $\cT_r$ is denoted $f_r(t)$. From the formula \eqref{burau equation}, it is easy to see that $f_r(t)=\frac{1-(-1)^r t^r}{1+t}$. We shall set $F_r(T)$ to denote the completed Alexander polynomial $\widehat{\Delta}_{\cT_r}(T)$.

\begin{proposition}\label{prop delta k alexander}
Let $n\geq 1$ be an integer and let $\underline{k}=(k_1, \dots, k_{n-1})\in \Z^{n-1}$, and let $b_{\underline{k}}$ denote the associated braid. Then, the Alexander polynomial $\Delta_{\underline{k}}(t):=\Delta_{b_{\underline{k}}}(t)$ is given by 
\[\Delta_{\underline{k}}(t)=\prod_{i=1}^{n-1} f_{k_i}(t).\]
\end{proposition}
\begin{proof}
We prove the result by induction on the number of strings. When $n=2$, there is nothing to prove, and thus assume that $n\geq 3$. We set $\underline{k}':=(k_2, \dots, k_{n-1})$. Let $b_1$ denote the braid on $2$-strings given by $\sigma_1^{k_1}$, and let $b_2$ denote the braid $b_{\underline{k}'}$. Then, it is easy to see that $b_{\uk}$ is the connected sum $b_1\sharp b_2$, where the second string in $b_1$ is added to the first string in $b_2$ so as to preserve orientation. We find that $\Delta_{b_1}(t)=f_{k_1}(t)$ and that $\Delta_{b_2}(t)=\Delta_{\uk'}(t)$. By inductive hypothesis, we find that $\Delta_{\uk'}(t)=\prod_{i=2}^{n-1}f_{k_i}(t)$. By a standard property of Alexander polynomials, $\Delta_{\uk}(t)=\Delta_{b_1\sharp b_2}(t)=\Delta_{b_1}(t)\Delta_{b_2}(t)=\prod_{i=1}^{n-1}f_{k_i}(t)$.
\end{proof}

\begin{proposition}
    Let $\uk=(k_1, \dots, k_{n-1})$ and $b_{\uk}\in \cF_n$ the associated braid, then, $|b_{\uk}|=\sum_{i=1}^{n-1} k_i$.  
\end{proposition}
\begin{proof}
Since $b_{\uk}$ is presented by $\sigma_1^{k_1}\dots \sigma_{n-1}^{k_{n-1}}$, it is clear that $|b_{\uk}|\leq \sum_{i=1}^{n-1} k_i$. On the other hand, if $b_{\uk}$ is represented by $\sigma_{i_1}^{a_1}\dots \sigma_{i_t}^{a_t}$, then by \eqref{burau equation},
\[\begin{split}
&\deg\left(\Delta_{\uk}(t)\right) \\
=&\deg\left(\op{det}\left(\op{I}_{n-1}-\bar{\psi}_n(b_{\uk})\right)\right)-n+1\\
=&\deg\left(\op{det}\left(\op{I}_{n-1}-\bar{\psi}_n(\sigma_{i_1}^{a_1}\dots \sigma_{i_t}^{a_t})\right)\right)-n+1 \\ 
\leq & \left(\sum_{i=1}^{n-1} a_i\right)-n+1.
\end{split}\]
On the other hand, we have that $\Delta_{\uk}(t)=\prod_{i=1}^{n-1} f_{k_i}(t)$ by Proposition \ref{prop delta k alexander}. We note that $\op{deg}f_{k_i}(t)=k_i-1$. Hence, we find that 
\[\deg \Delta_{\uk}(t)=\sum_{i=1}^{n-1} (k_i-1)=\left(\sum_{i=1}^{n-1} k_i\right)-n+1.\] Therefore, we find that 
\[\sum_{i=1}^{n-1} a_i\geq \sum_{i=1}^{n-1} k_i,\]hence, $|b_{\uk}|= \sum_{i=1}^{n-1} k_i$.
\end{proof}
\begin{proposition}\label{prop injectivity}
Given tuples $\uk=(k_1, \dots, k_{n-1})$ and $\uk'=(k_1', \dots, k_{n-1}')$, the braids $b=b_{\uk}$ and $b'=b_{\uk'}$ are isotopic if and only if $\uk=\uk'$. Thus, the association $\uk\mapsto b_{\uk}$ sets up a bijection from $\Z_{\geq 0}^{n-1}$ to $\cF_n$.
\end{proposition}
\begin{proof}
Suppose that $b_{\uk}$ and $b_{\uk'}$ are isotopic for $\uk$ and $\uk'$. Assume by way of contradiction that $\uk\neq \uk'$. Then, we find that $\psi_n(b_{\uk})=\psi_n(b_{\uk'})$. Recall that $U_i\in \op{GL}_{n-1}(\Lambda)$ is the matrix $\psi_n(\sigma_i)$, we find that
\[U_1^{k_1}\dots U_{n-1}^{k_{n-1}}=U_1^{k_1'}\dots U_{n-1}^{k_{n-1}'}.\]Let $j$ be the minimum value for which $k_j\neq k_j'$. Thus, $U_j^a\in \langle U_{j+1}, \dots, U_{n-1} \rangle $, where $a=k_j'-k_j$. Assume without loss of generality that $a>0$.

\par Recall that $U_j$ is given by  \[U_j:=\left( {\begin{array}{cccc}
 \op{I}_{j-1} & 0 & 0 & 0 \\
 0 & (1-t) & t & 0 \\
 0 & 1 & 0 & 0 \\
  0 & 0 & 0 & \op{I}_{n-j-1}, \\
 \end{array} } \right)\]
 and it is easy to see that the entry of $U_j^a$ in the $(j,j)$ position is a polynomial in $t$ of degree $a$. On the other hand, the entry in the $(j,j)$ position of any matrix in $\langle U_{j+1}, \dots, U_{n-1}\rangle$ is the constant polynomial $1$. Hence, $a=0$, which is a contradiction. We conclude that $\uk=\uk'$.
\end{proof}
\par Given a subset $\cF$ of $B_n$, and $x>0$, set 
\[\cF(x):=\left\{b\in \cF\mid |b|\leq x\right\}.\] Note that $\cF(x)=\cF\cap B_n(x)$, and is finite since $B_n(x)$ is a finite set. Given a subset $\cF'$ of $\cF_n$, the relative density of $\cF'$, is defined to be the following limit (provided the limit exists)
\[\mathfrak{d}(\cF'):=\lim_{x\rightarrow \infty} \frac{\# \cF'(x)}{\#\cF_n(x)}.\]
Given a function $H:\cF_n\rightarrow \mathbb{R}_{\geq 0}$, the \emph{average} value of $H$ on $\cF_n$ is defined as follows (provided the limit exists)
\[\op{Av}_H(\cF_n):=\lim_{x\rightarrow \infty} \left(\frac{1}{\# \cF_n(x)}\sum_{b\in \cF_n(x)} H(b)\right).\]
It follows from Proposition \ref{prop injectivity} that $\cF_n(x)$ can be identified with the set \[\{\uk=(k_1, \dots, k_{n-1})\in \Z_{\geq 0}^{n-1}\mid \sum_i k_i\leq x\},\] which is the set of lattice points in the finite polytope \begin{equation}\label{def of Px} P_x:= \left\{(z_1,\dots, z_{n-1})\in \mathbb{R}_{\geq 0}^{n-1}\mid \text{ such that }\sum_i z_i\leq x\right\}.\end{equation}

\par Given $\uk\in \Z_{\geq 0}^{n-1}$, we set $\mu_{p,\uk}$ and $\lambda_{p,\uk}$ to denote the invariants $\mu$ and $\lambda$-invariants of the link associated to the braid $b_{\uk}$.

\begin{lemma}\label{mu=0 lemma}
Let $n\geq 2$, and $\uk=(k_1, \dots, k_{n-1})\in \Z_{\geq 0}^{n-1}$. Then the following assertions hold
\begin{enumerate}
\item\label{p1 of mu=0 lemma} the Alexander polynomial $\Delta_{\uk}(t)$ is non-zero and $\mu_{\uk}=0$.
\item\label{p2 of mu=0 lemma} Setting $e(\uk):= \#\{k_i\mid k_i\text{ is even}\}$, we have that $\lambda_{\uk}\geq e(\uk)$. Furthermore, suppose that $p\nmid k_i$ for all $k_i$ that are even, then, $\lambda_{\uk}=e(\uk)$.
\end{enumerate}
\end{lemma}
\begin{proof}
We find that \[\begin{split}  \widehat{\Delta}_{\uk}(T)= & \prod_{i=1}^{n-1}F_{k_i}(T)\\ 
= & \frac{\prod_{i=1}^{n-1}\left(1-(-1)^{k_i}(1+T)^{k_i}\right)}{(2+T)^{n-1}} \\
= & \frac{\prod_{i=1}^{n-1}\left((1-(-1)^{k_i})-(-1)^{k_i}\sum_{j=1}^{k_i} \binom{k_i}{j} T^j\right)}{(2+T)^{n-1}} .\end{split}\]
It is easy to see that $\widehat{\Delta}_{\uk}(T)$ is a monic polynomial. Note that $p^{\mu_{\uk}}$ divides all coefficients of $\widehat{\Delta}_{\uk}(T)$ and hence, $\mu_{\uk}=0$. This proves \eqref{p1 of mu=0 lemma}.

\par Note that $T$ divides $F_{k_i}(T)$ if and only if $k_i$ is even. On the other hand, $T$ does not divide $(2+T)$. We thus find that $T^{e({\uk})}$ divides $\widehat{\Delta}_{\uk}(T)$. Write $\widehat{\Delta}_{\uk}(T)=u(T) g(T)$, where, $g(T)$ is a distinguished polynomial and $u(T)$ is a unit in $\Lambda$. Note that $\lambda_{\uk}$ is a the degree of $g(T)$, and $g(T)$ is divisible by $T^e$ and therefore, $\lambda_{\uk}\geq e$. 
\par Since $p$ is odd, we find that $(2+T)$ is a unit in $\widehat{\Lambda}$, and hence, up to a unit in $\widehat{\Lambda}$, $\widehat{\Delta}_{\uk}(T)$ is equal to \[\begin{split}& \prod_{i=1}^{n-1}\left((1-(-1)^{k_i})-(-1)^{k_i}\sum_{j=1}^{k_i} \binom{k_i}{j} T^j\right).\end{split}\]
If $k_i$ is odd, the constant term of $(1-(-1)^{k_i})-(-1)^{k_i}\sum_{j=1}^{k_i} \binom{k_i}{j} T^j$ is $2$, which is a unit in $\Z_p$. If $k_i$ is even, we find that 
\[(1-(-1)^{k_i})-(-1)^{k_i}\sum_{j=1}^{k_i} \binom{k_i}{j} T^j=-k_i T+\text{higher order terms}.\]
Therefore, 
\[\widehat{\Delta}_{\uk}(T)=u a T^e+\text{ higher order terms},\] where $u$ is a unit in $\Z_p$, and $a$ is the product of all $k_i$ such that $k_i$ is even. Therefore, by the discussion following \eqref{lambda min}, we have that $\lambda_{\uk}=e$ if and only if $p\nmid a$, i.e., if and only if $p\nmid k_i$ for all $k_i$ that are even. This completes the proof of \eqref{p2 of mu=0 lemma}.
\end{proof}
According to Lemma \ref{mu=0 lemma}, the $\mu$-invariant is equal to $0$ in the family of braids $\cF_n$. We study the distribution of the $\lambda$-invariant. In greater detail, for $\lambda\in \Z_{\geq 0}$, we study the limit
\begin{equation}\label{definition of d_n}\mathfrak{d}_n(\lambda):=\lim_{x\rightarrow \infty}\left(\frac{\#\left\{ \uk=(k_1, \dots, k_{n-1})\in \Z_{\geq 0}^{n-1}\mid \sum_i k_i\leq x\text{ and }\lambda_{\uk}=\lambda\right\}}{\#\left\{\uk=(k_1, \dots, k_{n-1})\in \Z_{\geq 0}^{n-1}\mid \sum_i k_i\leq x\right\}}\right).\end{equation}

\par In particular, we shall show that the above limit exists for each $\lambda\in \Z_{\geq 0}$. Let $H_r(T)$ denote the polynomial in $\widehat{\Lambda}$ given by \[ \begin{split} H_r(T):=& (2+T)F_r(T)=1-(-1)^r(1+T)^r \\ = & (1-(-1)^r)-(-1)^r\sum_{i=1}^r \binom{r}{i} T^i. \end{split}\]
Given a power series $H(T)\in \widehat{\Lambda}$, we set $\lambda(H)$ to denote the $\lambda$-invariant of $H$. Note that since $p$ is odd, $\lambda(F_r)=\lambda(H_r)$, since $(2+T)$ is a unit in $\widehat{\Lambda}$. Since $\widehat{\Delta}_{\uk}(T)=\prod_{i=1}^{n-1}F_{k_i}(T)$, we find that $\lambda_{\uk} = \sum_{i=1}^{n-1}\lambda(F_{k_i})$. We start by computing the following limits 
\[\begin{split}
    & \mathfrak{d}_ \lambda(F_\ast):=\lim_{x\rightarrow \infty} \frac{\#\left\{r\mid r\leq x, \lambda(F_r)=\lambda \right\}}{\#\left\{r\mid r\leq x\right\}}.
\end{split}\]

\begin{lemma}\label{lambda invariant of F_r} Suppose that $r$ is odd, then, $\lambda_p(F_r)=0$. On the other hand, if $r$ is even, write $r=p^{v_p(r)}r'$, we have that $\lambda_p(F_r)=p^{v_p(r)}$.
\end{lemma}
\begin{proof}
Note that since $p$ is odd, $\lambda(F_r)=\lambda(H_r)$. The constant coefficient of $H_r$ is given by $(1-(-1)^r)$. Suppose that $r$ is odd. Then, the constant coefficient of $H_r(T)$ is $2$, which is a $p$-adic unit. Therefore, $H_r(T)$ is a unit in the formal power series ring $\widehat{\Lambda}$, and therefore, $\lambda_p(F_r)=0$ when $r$ is odd. 

\par Next, suppose that $r$ is even. Then, the constant coefficient of $H_r(T)$ is zero, hence $\lambda(H_r)\geq 1$. Set $v=v_p(r)$, we have that $r=p^v r'$, where, $r'$ is even and coprime to $p$. We find that 
\[H_r(T)=1-(1+T)^r\equiv 1-(1+T^{p^v})^{r'}\equiv -r'T^{p^v}-\binom{r'}{2}T^{2p^v}+\dots \mod{p}.\]
Since $r'$ is not divisible by $p$, we find that $\lambda_p(H_r)=p^v$.

\end{proof}

\begin{proposition}\label{density propn}
The density $\mathfrak{d}_\lambda(F_\ast)$ is given as follows
\[\mathfrak{d}_\lambda(F_\ast)=\begin{cases}
\frac{1}{2}&\text{ if } \lambda=0,\\
\frac{1}{2p^i}\left(1-\frac{1}{p}\right)&\text{ if }\lambda=p^i,\\
0 & \text{ otherwise.}
\end{cases}\]
\end{proposition}
\begin{proof}
It follows from Lemma \ref{lambda invariant of F_r} that either $\lambda_p(F_r)=0$, or is a power of $p$. Therefore, if $\lambda\notin \{0\}\cup \{p^i\mid i\geq 0\}$, then, $\mathfrak{d}_{\lambda}(F_\ast)=0$. Furthermore, $\lambda_p(F_r)=0$ if and only if $r$ is odd, and thus,
\[\mathfrak{d}_0(F_\ast)=\lim_{x\rightarrow \infty} \frac{\#\left\{r\mid r\leq x, r\text{ is odd} \right\}}{\#\left\{r\mid r\leq x\right\}}=\frac{1}{2}. \]
It follows from Lemma \ref{lambda invariant of F_r} that $\lambda_p(F_r)=p^i$ if and only if $r\equiv 2ap^i\mod{2p^{i+1}}$, where $a\in \{1, \dots, p-1\}$. Thus, there are $(p-1)$ distinct congruence classes modulo $2p^{i+1}$. We find that
\[\begin{split}\mathfrak{d}_{p^i}(F_\ast)=& \sum_{a=1}^{p-1}  \lim_{x\rightarrow \infty} \frac{\#\left\{r\mid r\leq x, r\equiv 2ap^i\mod{p^{i+1}} \right\}}{\#\left\{r\mid r\leq x\right\}} \\
= & \sum_{a=1}^{p-1} \frac{1}{2p^{i+1}}=\frac{1}{2p^i}\left(1-\frac{1}{p}\right).\end{split}\]
\end{proof}
\par Let $\ulambda=(\lambda_1,\dots, \lambda_{n-1})$ be a tuple of non-negative integers. Set $\mathfrak{d}(\ulambda)$ denote the limit defined as follows
\[\mathfrak{d}(\ulambda):=\lim_{x\rightarrow \infty}\left(\frac{\#\left\{ \uk=(k_1, \dots, k_{n-1})\in \Z_{\geq 0}^{n-1}\mid \sum_i k_i\leq x\text{ and }\lambda(F_{k_i})=\lambda_i \text{ for }1\leq i\leq n-1\right\}}{\#\left\{\uk=(k_1, \dots, k_{n-1})\in \Z_{\geq 0}^{n-1}\mid \sum_i k_i\leq x\right\}}\right).\]

 We shall in due course evaluate the above limits for all choices of $\ulambda\in \Z_{\geq 0}^{n-1}$. However, we first relate the above limits to $\mathfrak{d}_n(\lambda)$. We recall that, given an integer $\lambda\in \Z_{\geq 0}$, a composition of $\lambda$ into $(n-1)$ parts is a tuple of non-negative integers $\ulambda=(\lambda_1, \dots, \lambda_{n-1})$ such that $\sum_{i=1}^{n-1}\lambda_i=\ulambda$. 
 
 \begin{definition}\label{def A}
 Set $\mathcal{A}_{(n-1),\lambda}$ to be the set of all compositions $\ulambda=(\lambda_1, \dots, \lambda_{n-1})$ of $\lambda$ such that each entry $\lambda_i$ is either $0,1$ or a positive power of $p$.
 \end{definition} It follows from Lemma \ref{lambda invariant of F_r} that if $\lambda_i$ is not in the set $\{0\}\cup \{p^i\mid i\geq 0\}$, for all entries $1\leq i\leq n-1$, then, $\mathfrak{d}(\ulambda)=0$. This is because the only values that $\lambda(F_{k_i})$ can take must be in the set $\{0\}\cup \{p^i\mid i\geq 0\}$. It is easy to see that the following relation holds
 \begin{equation}\label{boring equation}
\mathfrak{d}_n(\lambda)=\sum_{\ulambda\in \mathcal{A}_{(n-1), \lambda}} \mathfrak{d}(\ulambda).
 \end{equation}

 Therefore, in order to evaluate $\mathfrak{d}_n(\lambda)$, it suffices to evaluate $\mathfrak{d}(\ulambda)$ for all $\ulambda\in \mathcal{A}_{(n-1), \lambda}$. In the following discussion, the tuple $\ulambda$ is fixed. Given a region $D$ in $\mathbb{R}_{\geq 0}^{n-1}$, set \[N_{\underline{\lambda}}(D):=\#\{\uk=(k_1, \dots, k_{n-1})\in D\cap(\Z_{\geq 0}^{n-1})\mid \lambda(F_{k_i})=\lambda_i\text{ for all }i\}.\] Given a point $z=(z_1, \dots, z_{n-1})\in \Z_{\geq 0}^{n-1}$, we set $R_z$ to denote the rectangle defined by
\begin{equation}\label{def of Rz}R_z=\left\{(z_1+a_1, \dots, z_i+a_i,\dots,  z_{n-1}+a_{n-1})\mid 0\leq a_i < p^{\lambda_i+1}\text{ for }i=1, \dots, n-1\right\}.\end{equation}
Although this is suppressed in our notation, the definition of $R_z$ depends on the choice of $\ulambda$. 
\begin{lemma}\label{N lambda lemma}
With respect to notation above, we have that \[N_{\underline{\lambda}}(R_z)=\op{Vol}(R_z) \prod_{i} \mathfrak{d}_{\lambda_i}(F_\ast).\]
\end{lemma}
\begin{proof}
The proof of the above result is similar to that of Proposition \ref{density propn}, and is omitted.
\end{proof}

\begin{lemma}\label{lemma product lambda}
Let $n\geq 2$ and $\ulambda=(\lambda_1, \dots, \lambda_{n-1})\in \Z_{\geq 0}^{n-1}$. Then, the following relation holds
\[\mathfrak{d}(\ulambda)=\prod_{i=1}^{n-1} \mathfrak{d}_{\lambda_i}(F_\ast).\]
\end{lemma}
\begin{proof}
Let $x>0$ be a real number, and recall that $P_x$ is the region in $\mathbb{R}_{\geq 0}^{n-1}$ given by \eqref{def of Px}. Recall from \eqref{definition of d_n} that 
\[\begin{split}\mathfrak{d}(\ulambda)=& \lim_{x\rightarrow\infty} \frac{\#\left\{\uk\in P_x\cap (\Z_{\geq 0}^{n-1})\mid \lambda_{\uk}=\ulambda\right\}}{\# P_x\cap (\Z_{\geq 0}^{n-1})} \\
=& \lim_{x\rightarrow\infty} \frac{\#\left\{\uk\in P_x\cap (\Z_{\geq 0}^{n-1})\mid \lambda_{\uk}=\ulambda\right\}}{\op{Vol}(P_x)}. \end{split}\]
In the above equality, we have simply used that as a function of $x$, 
\[\lim_{x\rightarrow \infty} \frac{\# P_x\cap (\Z_{\geq 0}^{n-1})}{\op{Vol}(P_x)}=1,\] the proof of which is strightforward, and thus omitted. For $i$ in the range $1\leq i\leq (n-1)$, let $e_i$ denote the coordinate vector with $0$ in all positions $j\neq i$ and $1$ in the $i$-th position. Let $\tilde{L}$ be the lattice in $\mathbb{R}^{n-1}$ generated by $p^{\lambda_i+1}e_i$ for $i=1, \dots, (n-1)$, and set $L:=\tilde{L}\cap \mathbb{R}_{\geq 0}^{n-1}$. For each point $z=(z_1, \dots, z_{n-1})$ in $L$, let $R_z$ be the rectangle defined by \eqref{def of Rz}. Note that the rectangles $R_z$ are all mutually disjoint as $z$ ranges over $L$. Furthermore, these rectangles cover $\mathbb{R}_{\geq 0}^{n-1}$.

\par We set $M_1(x):=\left\{z\in L\mid R_z\subseteq P_x\right\}$ and $M_2(x):=L\cap P_x$. For $i=1,2$, we set $D_i(x)$ to denote the union 
\[D_i(x):=\bigcup_{z\in M_i(x)} R_z.\] Note that $M_1(x)$ is contained in $M_2(x)$ and that $D_1(x)\subseteq P_x\subseteq D_2(x)$. Set $\Omega(x):=M_2(x)\backslash M_1(x)$, and we find that 
\[D_2(x)\backslash D_1(x)=\bigcup_{z\in \Omega(x)} R_z.\]
Let $t=(t_1, \dots, t_{n-1})\in D_2(x)$, thus, there is a unique point $z\in M_2(x)=L\cap P_x$ such that $t\in R_z$. We find that $t_i<z_i+p^{\lambda_i+1}$ for all $i$. Setting $C:=\sum_{i=1}^{n-1} p^{\lambda_i+1}$, we find that 
\[\sum_{i=1}^{n-1} t_i< \sum_{i=1}^{n-1} z_i+C\leq x+C.\]
Therefore, $t\in P_{x+C}$, and hence, we have shown that $D_2(x)\subseteq P_{x+C}$. Now suppose that $t\notin D_1(x)$. Then, we find that $R_z$ is not entirely contained in $P_x$, and hence $\sum_{i=1}^{n-1} (z_i+p^{\lambda_i+1})>x$. Therefore,
\[\sum_{i=1}^{n-1} t_i \geq \sum_{i=1}^{n-1} z_i= \sum_{i=1}^{n-1} (z_i+p^{\lambda_i+1})-C>x-C,\] i.e., $t\notin P_{x-C}$. We have therefore shown that $D_2(x)\backslash D_1(x)$ is contained in $P_{x+C}\backslash P_{x-C}$. A simple application of multivariable calculus shows that 
\[\op{Vol}(P_x)=\frac{1}{(n-1)!} x^{n-1}.\]We find that 
\[\lim_{x\rightarrow \infty} \frac{\op{Vol}(D_2(x))-\op{Vol}(D_1(x))}{\op{Vol}(P_x)}\leq \lim_{x\rightarrow \infty} \frac{\op{Vol}(P_{x+C})-\op{Vol}(P_{x-C})}{\op{Vol}(P_x)}=0.\]
We note that 
\[\begin{split}1\geq  \frac{\op{Vol}(D_1(x))}{\op{Vol}(P_x)}= &\frac{\op{Vol}(D_2(x))}{\op{Vol}(P_x)}-\left(\frac{\op{Vol}(D_2(x))-\op{Vol}(D_1(x))}{\op{Vol}(P_x)}\right) \\
\geq & 1- \left(\frac{\op{Vol}(D_2(x))-\op{Vol}(D_1(x))}{\op{Vol}(P_x)}\right).\end{split}\]
Therefore, 
\[\lim_{x\rightarrow \infty} \left|1- \frac{\op{Vol}(D_1(x))}{\op{Vol}(P_x)}\right|=0,\] and thus, 
\[\lim_{x\rightarrow \infty} \frac{\op{Vol}(D_1(x))}{\op{Vol}(P_x)} =1.\] Similar arguments show that 
\[\lim_{x\rightarrow \infty} \frac{\op{Vol}(D_2(x))}{\op{Vol}(P_x)} =1.\]
Recall that $D_i(x)$ is a disjoint union of rectangles $R_z$, as $z$ ranges over the finite set $M_i(x)$. It thus follows from Lemma \ref{N lambda lemma} 
that \[N_{\underline{\lambda}}(D_i(x))=\sum_{z\in M_i(x)}\op{Vol}(R_z) \prod_{i=1}^{n-1} \mathfrak{d}_{\lambda_i}(F_\ast)=\op{Vol}(D_i(x))\prod_{i=1}^{n-1} \mathfrak{d}_{\lambda_i}(F_\ast).\]
Therefore, we find that for $i=1,2$,
\[\lim_{x\rightarrow \infty} \frac{N_{\underline{\lambda}}(D_i(x))}{\op{Vol}(P_x)}=\left(\lim_{x\rightarrow \infty} \frac{\op{Vol}(D_i(x))}{\op{Vol}(P_x)}\right)\prod_{i=1}^{n-1} \mathfrak{d}_{\lambda_i}(F_\ast)=\prod_{i=1}^{n-1} \mathfrak{d}_{\lambda_i}(F_\ast).\]
We have that 
\[\frac{N_{\underline{\lambda}}(D_1(x))}{\op{Vol}(P_x)}\leq \frac{N_{\underline{\lambda}}(P_x)}{\op{Vol}(P_x)} \leq \frac{N_{\underline{\lambda}}(D_2(x))}{\op{Vol}(P_x)},\]
and 
\[\lim_{x\rightarrow \infty} \frac{N_{\underline{\lambda}}(D_1(x))}{\op{Vol}(P_x)}=\lim_{x\rightarrow \infty} \frac{N_{\underline{\lambda}}(D_2(x))}{\op{Vol}(P_x)}=\prod_{i=1}^{n-1} \mathfrak{d}_{\lambda_i}(F_\ast).\]
Thus, by the sandwich theorem for limits, 
\[\mathfrak{d}(\ulambda)=\lim_{x\rightarrow \infty} \left( \frac{N_{\underline{\lambda}}(P_x)}{\op{Vol}(P_x)} \right)=\prod_{i=1}^{n-1} \mathfrak{d}_{\lambda_i}(F_\ast).\]
\end{proof}
\par We present the proof of the main result of the manuscript.
\begin{proof}[Proof of Theorem \ref{main thm}]
\par The assertion \eqref{main thm p1} of the theorem follows from part \eqref{p1 of mu=0 lemma} of Lemma \ref{mu=0 lemma}. We prove part \eqref{main thm p2} of the theorem. Assume that $\lambda>0$. Let $\mathcal{A}_{n-1, \lambda}$ be the set of compositions from Definition \ref{def A}. Then we recall from \eqref{boring equation} that \[\mathfrak{d}_n(\lambda)=\sum_{\ulambda\in \mathcal{A}_{(n-1), \lambda}} \mathfrak{d}(\ulambda).\]
For $j$ in the range $1\leq j\leq n-1$, we let $\mathcal{A}_{n-1,\lambda}^{(j)}$ be the subset of $\mathcal{A}_{n-1,\lambda}$ consisting of compositions $\alpha=(\alpha_1, \dots, \alpha_{n-1})$ such that exactly $j$ entries $\alpha_i$ are non-zero. Let 
\[\pi_j:\mathcal{A}_{n-1,\lambda}^{(j)}\rightarrow \mathcal{C}_{j, \lambda} \] be the map which has the effect of deleting all entries that are $0$. In greater detail, given $\alpha=(\alpha_1, \dots, \alpha_{n-1})\in \mathcal{A}_{n-1,\lambda}^{(j)}$, let $\alpha_{i_1}, \dots, \alpha_{i_j}$ be the non-zero entries ordered such that $i_1<i_2<i_3<\dots <i_{j}$. We may express $\alpha_{i_k}=p^{a_k}$, and note that the tuple $(p^{a_1}, \dots, p^{a_j})$ is a composition in $\mathcal{C}_{j, \lambda}$. We set $\pi_j(\alpha):=(p^{a_1}, \dots, p^{a_j})$.
\par The map $\pi_j$ is clearly surjective, and the fibers are easily seen to have cardinality equal to $\binom{n-1}{j}$. We thus note that
\begin{equation}\label{last equation}
\#\mathcal{A}_{n-1, \lambda}^{(j)}=\binom{n-1}{j}\#\mathcal{C}_{j,\lambda}
\end{equation}The relation \eqref{boring equation} can be rewritten as follows
\[\mathfrak{d}_n(\lambda)=\sum_{j=1}^{n-1} \left(\sum_{\ulambda\in \mathcal{A}_{n-1, \lambda}^{(j)}} \mathfrak{d}(\ulambda)\right).\]

Let $\ulambda\in \mathcal{A}_{n-1, \lambda}^{(j)}$, then, by Lemma \ref{lemma product lambda}, we have that \[\mathfrak{d}(\ulambda)=\prod_{i=1}^{n-1} \mathfrak{d}_{\lambda_i}(F_\ast).\]
It follows from Proposition \ref{density propn} that for $\ulambda\in \mathcal{A}_{n-1, \lambda}^{(j)}$,
\[\mathfrak{d}(\ulambda)=\frac{1}{2^{n-1}p^{\lambda}}\left(1-\frac{1}{p}\right)^j.\] Therefore, we find that 
\[\begin{split}\mathfrak{d}_n(\lambda)= &  \frac{1}{2^{n-1}p^{\lambda}}\sum_{j=1}^{n-1}\left(1-\frac{1}{p}\right)^j \#\mathcal{A}_{n-1, \lambda}^{(j)}, \\ 
& \frac{1}{2^{n-1}p^\lambda}\sum_{j=1}^{n-1}\binom{n-1}{j}\left(1-\frac{1}{p}\right)^j\#\mathcal{C}_{j,\lambda},\end{split}\]
where the second equality follows from \eqref{last equation}.
\par The result in the case when $\lambda=0$ follows via a similar argument, which is omitted.
\end{proof}

\par \textbf{An Example:} We illustrate the above result through an explicit example. We let $p=5$ and $n=8$ and $\lambda=31$. We wish to evaluate $\mathfrak{d}_n(\lambda)=\mathfrak{d}_8(31)$ for the prime $p=5$. According to \eqref{main thm eqn}, we have that
\[\mathfrak{d}_8(31)=\frac{1}{2^{7}5^{31}}\sum_{j=1}^{7}\binom{7}{j}\left(1-\frac{1}{5}\right)^j\#\mathcal{C}_{j,31}.\]
We evaluate the numbers $\#\mathcal{C}_{j,\lambda}$ for $j$ in the range $0\leq j \leq 7$, 
\begin{itemize}
\item $j=7$: There are $14$ compositions in total, namely
\[\begin{split}
& (1,5,5,5,5,5,5),(5,1,5,5,5,5,5),(5,5,1,5,5,5,5),(5,5,5,1,5,5,5),\\ &(5,5,5,5,1,5,5),(5,5,5,5,5,1,5), (5,5,5,5,5,5,1),\\
& (5^2,1,1,1,1,1,1),(1,5^2,1,1,1,1,1),(1,1,5^2,1,1,1,1), (1,1,1,5^2,1,1,1),\\
& (1,1,1,1,5^2,1,1),(1,1,1,1,1,5^2,1),(1,1,1,1,1,1,5^2).
\end{split}
\]
\item j=3: There are $6$ compositions, namely
\[\begin{split}
& (5^2,5,1), (5^2, 1,5), (5, 5^2, 1), (5, 1,5^2), (1, 5, 5^2), (1, 5^2, 5).
\end{split}\]
For the other values of $j$, there are no compositions. Hence, we find that 
\[\mathfrak{d}_8(31)=\frac{1}{2^75^{31}}\left(6\binom{7}{3}(4/5)^3+14 \binom{7}{7}(4/5)^7\right).\]
\end{itemize}

\bibliographystyle{alpha}
\bibliography{references}
\end{document}